\documentclass[12pt]{amsart}

\usepackage{amsmath}
\usepackage{amssymb}
\usepackage{amsfonts}
\usepackage{amsthm}
\usepackage{enumerate}
\usepackage{hyperref}
\usepackage{color}
\usepackage[all]{xy}
\usepackage{tikz}

\theoremstyle{plain}
\newtheorem{thm}{Theorem}[section]

\newtheorem{lem}[thm]{Lemma}

\newtheorem{clm}[thm]{Claim}

\begin{document}
% ------------------------------------------------------------------------
\title{Convex Hull of Face Vectors of Colored Complexes}

% ------------------------------------------------------------------------
\author{Afshin Goodarzi}
\thanks{}

\address{Royal Institute of Technology, Department of Mathematics, S-100 44, Stockholm, Sweden}
\email{afshingo@kth.se}
\maketitle

% ------------------------------------------------------------------------
\begin{abstract}
 In this paper we verify a conjecture by Kozlov (Discrete Comput Geom 18 (1997) 421--431), 
 which describes the convex hull of the set of face vectors of $r$-colorable complexes on 
 $n$ vertices. As part of the proof we derive a generalization of Tur\'{a}n's graph theorem.

\end{abstract}
% ------------------------------------------------------------------------

\section{Introduction}

Let $\Delta$ be a simplicial complex on $n$ vertices and let $\Delta_k$ be the set of all
faces of $\Delta$ of cardinality $k$. The face vector of $\Delta$ is 
$f(\Delta)=(n, f_2, f_3, \ldots)$ where $f_k$ is the cardinality of $\Delta_k$.
A simplicial complex $\Delta$ is said to be $r$-colorable if its underlying graph
(i.e., the graph with the same vertices as $\Delta$ and with edges $\Delta_2$) is
$r$-colorable.

Throughout this paper, by a graph $G$ we mean a finite graph without any loops or
multiple edges. The set of vertices and edges of $G$ will be denoted by $V(G)$ and 
$E(G)$, respectively. The cardinality of $V(G)$ and $E(G)$ are \emph{order} and 
\emph{size} of $G$. A $k$-\emph{clique} in $G$ is a complete induced subgraph of $G$ of
order $k$. The \emph{clique vector} of $G$ is $c(G)=(c_1(G), c_2(G), \ldots)$, where 
$c_k(G)$ is the number of $k$-cliques in $G$. The \emph{Tur\'{a}n graph} $T(n,r)$ is the
complete $r$-partite graph of order $n$ with cardinality of the maximal independent sets
``as equal as possible''.

A vector $g\in\mathbb{R}^d$ will be called positive if it has positive coordinates. 
The $k$-\emph{truncation} of $g$, denoted by $g^k$, is the vector whose first $k$
coordinates are equal to the coordinates of $g$, and the rest are equal to zero, for $k= 1, 2, \ldots, d$.

Kozlov conjectured~\cite[Conjecture 6.2]{K} that the convex hull of the face vectors of
$r$-colorable complexes on $n$ vertices has a simple description in term of the clique vector of the Tur\'{a}n graph. The main result of this paper is to show the validity of his conjecture, more precisely:

\begin{thm}\label{main}
The convex hull of $f$-vectors of $r$-colorable complexes on $n$ vertices is generated
by the truncations of the clique vector of Tur\'{a}n graph $T(n,r)$.
\end{thm}

The structure of the paper is as follows. In Section~\ref{3}, we set up a method for 
finding the convex hull of the skeleta of a positive vector. The generalization of 
Tur\'{a}n's graph theorem will be proved in Section~\ref{4}. Finally, 
in Section~\ref{5} we will prove our main result.

\section{Thales' Lemma}\label{3}

Let $g=(g_1, \ldots, g_d)$ be a positive vector in $\mathbb{R}^d$ and denote by 
$\mathcal{C}_g$ the convex hull generated by the origin and all truncations of $g$. 
If $g\in\mathbb{R}^2$, then $\mathcal{C}_g$ is the boundary and interior of a right 
angle triangle. In this case using Thales' Intercept theorem, one can see that a 
positive vector $(a,b)$ is in $\mathcal{C}_g$ if and only if $a\leq g_1$ and $(b/a)\leq 
(g_2/g_1)$. The following result is a generalization of this simple observation.

\begin{lem}\label{th}
Let  $g=(g_1, \ldots, g_d)$ and $f=(f_1, \ldots, f_d)$ be two positive vectors. Then 
$f\in\mathcal{C}_g$ if and only if $f_1\leq g_1$ and $f_ig_j\leq f_jg_i$ for all 
$1\leq j< i\leq d$.
\end{lem}

\begin{proof}
The vectors $g^1, \ldots, g^d$ form a basis for $\mathbb{R}^d$. So there exists 
$c=(c_1, \ldots, c_d)\in \mathbb{R}^d$ such that $f=\sum_1^d c_ig^i$. So we have
\begin{eqnarray*}
f_d&=&c_d g_d, \\
f_{d-1}&=& (c_{d-1}+c_d) g_{d-1},\\
&\vdots&\\
f_1&=&(c_1+\ldots+c_d)g_1.
\end{eqnarray*}
On the other hand, $f\in\mathcal{C}_g$ if and only if $c_j\geq 0$ for all $j$ and 
$\sum c_i\leq 1$. Therefore we have $f\in\mathcal{C}_g$ if and only if 
$f_1=(\sum c_i) g_1 \leq g_1$ and 
$f_ig_j= (c_i+\ldots +c_d) g_ig_j\leq (c_j+\ldots +c_i+\ldots +c_d) g_jg_i=f_jf_i$.
\end{proof}

In the special case where $g$ is the face vector of the $(n-1)$-dimensional simplex, the
result above is already contained in the work of Kozlov~\cite[Section 5]{K}. His proof, 
however, works in the general case as well. 

\section{Tur\'{a}n Graphs}\label{4}
Let us denote by $\mathcal{G}(n,r)$ the set of all graphs $G$ of order $n$ and clique 
number $\omega(G)\leq r$. Tur\'{a}n graph has many extremal behaviors among all graphs 
in $\mathcal{G}(n,r)$. Recall that \emph{Tur\'{a}n graph} $T(n,r)$ is the complete 
$r$-partite graph of order $n$ with cardinality of the maximal independent sets as 
equal as possible. We will denote by $t_k(n,r)$ the number of $k$-cliques in $T(n,r)$.

In 1941 Tur\'{a}n proved that among all graphs in $\mathcal{G}(n,r)$, the Tur\'{a}n 
graph $T(n,r)$ has the maximum number of edges. This result, \emph{Tur\'{a}n's graph 
theorem}, is a cornerstone of Extremal Graph Theory. There are many different and 
elegant proofs of Tur\'{a}n's graph theorem. Some of these proofs were discussed 
in~\cite{A} and in~\cite[Chapter 36]{AZ}.

Later, in 1949, Zykov generalized Tur\'{a}n's graph theorem by showing that 
$c_k(G)\leq t_k(n,r)$ for all $G\in\mathcal{G}(n,r)$ and all $k$. Here we state and 
prove a generalization of Zykov's result.

\begin{thm}\label{cv}
For any graph $G\in\mathcal{G}(n,r)$ and for each $k\in\{2, \ldots, r\}$, one has
$$\frac{c_r(G)}{t_r(n,r)}\leq\ldots\leq\frac{c_k(G)}{t_k(n,r)}\leq
\frac{c_{k-1}(G)}{t_{k-1}(n,r)}\leq\ldots\leq\frac{c_2(G)}{t_2(n,r)}\leq 1.$$
\end{thm}

\begin{proof}
Let $G\in\mathcal{G}(n,r)$. We may assume that $G$ is not complete and and for a fixed 
$k$, $q_k(G):=c_k(G)/ c_{k-1}(G)$ is maximum among all graphs in $\mathcal{G}(n,r)$. 
Let $u$ and $v$ be two disconnected vertices in $G$ and define $G_{u\rightarrow v}$
to be the graph with the same vertex set as $G$ and with edges 
$E(G_{u\rightarrow v})=\left(E(G)\cup(\cup_{w\in N(v)} \{u,w\})\right)\setminus 
(\cup_{z\in N(u)} \{u,z\})$.\\
The following properties can be simply verified
\begin{itemize}
\item[$\bullet$] $G_{u\rightarrow v}\in\mathcal{G}(n,r)$,
\item[$\bullet$] $c_k(G_{u\rightarrow v})=c_k(G)-c_{k-1}(G[N(u)])+c_{k-1}(G[N(v)])$.
\end{itemize}
On the other hand, it is straightforward to check that either one of 
$q_k(G_{u\rightarrow v})$ and $q_k(G_{v\rightarrow u})$ is strictly greater than 
$q_k(G)$, or they are all equal. Hence $q_k(G_{u\rightarrow v})$ is maximal.\\
Now consider all vertices of $G$ that are not connected to $v$. let us label them by 
$u_1, \ldots, u_m$. We define 
$$G^1:= G_{u_1\rightarrow v}, \ldots, G^j:=G^{j-1}_{u_j\rightarrow v}, 
\ldots, G^m:=G^{m-1}_{u_m\rightarrow v}.$$
If $G^m\setminus\{v, u_1, \ldots, u_m\}$ is a clique, then we stop. 
If not, there exists a vertex $w\in G^m\setminus\{v, u_1, \ldots, u_m\}$ which is 
not connected to all other vertices. We repeat the above process with $w$ and continue 
until the remaining vertices form a clique. So we will obtain a complete multipartite 
graph $H\in \mathcal{G}(n,r)$ such that $q_k(H)$ is maximum. If $H$ is a 
Tur\'{a}n graph, then we are done. If not there exist two maximal independent sets 
$I_1=\{w_1, \ldots, w_m\}$ and $I_2=\{z_1, \ldots, z_l\}$ such that $m-2\geq l$. 
Let $H'$ be the graph obtained by removing all edges of the form $w_mz_i$ and adding 
new edges $w_mw_i$ for all $1\leq i\leq l$. Then it is easy to see that for all $j$, 
$H'$ has as many $j$-cliques as $H$ has and, in particular $q_k(H')$ is maximum. 
Therefore $q_k(H'_{w_m\rightarrow z_1})$ is maximum as well and the result follows by
repeating the above process.

\end{proof}

\textbf{Remark 3.2.} The operator $G_{u\rightarrow v}$ in our proof is similar to operators
used in~\cite[p. 238]{AZ} and in~\cite[Theorem 3.3]{K}. However it may belong to 
``folklore'' graph theory, since its origin is not clear.

\section{Proof of Theorem~\ref{main}}\label{5}
In order to prove our main result, using Thales' Lemma~\ref{th}, it is enough to 
show that for any $r$-colorable complex $\Delta$ on $n$ vertices and for each $k$,
\begin{eqnarray*}\label{ineq}
f_k(\Delta)/f_{k-1}(\Delta)\leq t_k(n,r)/t_{k-1}(n,r).
\end{eqnarray*}
To prove inequalities above, we need further definitions.

Let $1\leq k\leq r$ be fixed integers and let us denote by $\mathbb{N}_i$ the set of
all positive integers whose residue modulo $r$ is equal to $i$. The set of all 
$r$-colored $k$-subsets is 
$$\mathcal{M}(k,r)=\left\{F\in {\mathbb{N} \choose k} \Big| \quad |F\cap\mathbb{N}_i|\leq 1
\text{ for all } i\right\}.$$ We consider the partial order $<_p$ on $\mathcal{M}(k,r)$
defined as follows. For $T=\{t_1,\ldots, t_k\}$ and $S=\{s_1,\ldots, s_k\}$ with
$t_1<\ldots<t_k$ and $s_1<\ldots <s_k$ in $\mathcal{M}(k,r)$, set $T<_p S$ if 
$t_i\leq s_i$ for every $1\leq i\leq k$. A family 
$\mathcal{F}\subseteq \mathcal{M}(k,r)$ is said to be $r$-\emph{color shifted} 
if whenever $S\in\mathcal{F}$, $T<_p S$, and $T\in\mathcal{M}(k,r)$ one has 
$T\in\mathcal{F}$.
A simplicial complex is said to be $r$-\emph{color shifted} if for any $k$ the set of
its $k$-faces is an $r$-color shifted family.
It is known that for any $r$-colorable complex $\Delta$ on $n$-vertices and for
any $k$ there exists a $r$-color shifted complex $\Gamma$ such that 
$f_k(\Delta)=f_k(\Gamma)$ and $f_{k-1}(\Delta)\geq f_{k-1}(\Gamma)$. 
(see \cite[Proposition 3.1]{FFK}, for instance.)

\begin{proof}
We use induction on $r$. For $r=1$, $\Delta$ is totally disconnected and the statement
clearly holds. Now assume that the statement holds for any $(r-1)$-colorable complex.
Fix a $k$ and let $\Delta$ be an $r$-colorable complex on $n$ vertices such that 
$$\frac{f_k(\Delta)}{f_{k-1}(\Delta)}=\max\left\{\frac{f_k(\Gamma)}{f_{k-1}(\Gamma)} 
\Big|\quad \Gamma\text{ is an $r$-colorable on $n$ vertices }\right\}.$$
We may assume that $\Delta$ is color shifted. We may also assume that for any $j\geq k$
if $\Delta$ contains the boundary of a $j$-simplex $\delta$, then $\Delta$ contains 
$\delta$ itself. Let $I_{(1)}=\{u_1, \ldots, u_{m-1}\}$ be the set of all vertices that
are not connected to the vertex $1$. For $u\in I_{(1)}$ define $\Delta_{u\rightarrow 1}$
to be the complex obtained by removing all faces which contain $\{u\}$ properly and 
adding new faces $F\cup\{u\}$ for all $F\in\mbox{link}_{\Delta} 1$. Note that if we have
an $r$-coloring of $\Delta$, it is possible that $u$ and a vertex in 
$\mbox{link}_{\Delta} 1$ has the same color, however we can change the color of $u$ with
the color of $1$, so this construction preserves $r$-colorability.\\
It is easy to see that 
$$f_j(\Delta_{u\rightarrow 1})=f_j(\Delta)-f_{j-1}(\mbox{link}_{\Delta} u)+
f_{j-1}(\mbox{link}_{\Delta} 1).$$
Hence $f_k(\Delta_{u\rightarrow 1})/f_{k-1}(\Delta_{u\rightarrow 1})$ is maximum
as well. So if we define 
$$\Lambda=(\ldots((\Delta_{u_1\rightarrow 1})_{u_2\rightarrow 1})\ldots)_{u_{m-1}\rightarrow 1},$$
then $\Lambda$ is $r$-colorable and $f_k(\Lambda)/f_{k-1}(\Lambda)$ is maximum,
since in each step our operator preserves $f_k/f_{k-1}$ and $r$-colorability.\\
Let us denote by $L$ and $D$, the subcomplex $\mbox{link}_{\Delta} 1$ and the subcomplex
of $\Delta$ induced by vertices of $\mbox{link}_{\Delta} 1$, respectively. 
It is easy to see that $f_j(\Lambda)=mf_{j-1}(L)+ f_j(D)$.

 \begin{clm} $D_j=L_j$, for any $j\geq k-1$.
 \end{clm}
 \begin{proof}
 It is easy to see that $L_j\subseteq D_j$. So assume that $F\in D_j$. 
 For any $u\in F$ we have $F\setminus\{u\}\cup\{1\}\in\Delta$, by the structure of 
 $\Delta$. Hence the boundary of $F\cup\{1\}$ is in $\Delta$ and we have 
 $F\cup\{1\}\in\Delta$, therefore $F\in L_j$.
\end{proof}
So we have 
$$\frac{f_k(\Lambda)}{f_{k-1}(\Lambda)}=\frac{mf_{k-1}(L)+ f_k(L)}{mf_{k-2}(L)+ 
f_{k-1}(L)}.$$
On the other hand, since $L$ is $(r-1)$-colorable, there exists a graph 
$H\in\mathcal{G}(|V(L)|,r-1)$ such that
$f_t(L)/f_{t-1}(L)\leq c_t(H)/c_{t-1}(H)$ for any $2\leq t\leq r-1$. 
Denote by $G^k$ the graph obtained by joining $H$ and a totally disconnected graph on 
$m$ vertices. Clearly $G^k\in\mathcal{G}(n,r)$ and we have $c_t(G^k)=mc_{t-1}(H)+c_t(H)$
for all $t$. So we have
\begin{eqnarray*}
c_{k-1}(G^k)f_k(\Lambda)&=& (mc_{k-2}(H)+c_{k-1}(H))(mf_{k-1}(L)+f_k(L))\\
&=& m^2c_{k-2}f_{k-1}(L)+ mc_{k-2}(H)f_k(L)+\\ 
& & mf_{k-1}(L)c_{k-1}(H)+ c_{k-1}(H)f_k(L)\\
&\leq& m^2c_{k-1}f_{k-2}(L)+ mc_{k}(H)f_{k-2}(L)+\\ 
& & mf_{k-1}(L)c_{k-1}(H)+ c_{k}(H)f_{k-1}(L)\\
&=& c_{k}(G^k)f_{k-1}(\Lambda).\\
\end{eqnarray*}

So we have proved that for any $r$-colorable simplicial complex on $n$ vertices and
for a fixed $k$ there exists a graph $G^k\in\mathcal{G}(n,r)$ such that 
$f_k(\Delta)/f_{k-1}(\Delta)\leq c_k(G^k)/c_{k-1}(G^k).$ 
On the other hand by using Theorem~\ref{cv}, for all $k$, we have
$$\frac{c_k(G^k)}{c_{k-1}(G^k)}\leq\frac{t_k(n,r)}{t_{k-1}(n,r)},$$
as desired.
\end{proof}

\subsection*{Acknowledgments}
The author would like to thank Bruno Benedetti and Anders Bj\"{o}rner for helpful suggestions and
discussions.

\end{document}